\documentclass[12pt]{amsart}

\usepackage{amsfonts,amsmath,amsthm,amscd,amssymb,latexsym}
\usepackage{enumerate}

\makeatletter
\@namedef{subjclassname@2010}{%
  \textup{2010} Mathematics Subject Classification}
\makeatother

\newtheorem{thm}{Theorem}[section]
\newtheorem{cor}[thm]{Corollary}
\newtheorem{lem}[thm]{Lemma}
\newtheorem{prop}[thm]{Proposition}

\newtheorem{question}[thm]{Question}

\theoremstyle{defin}
\newtheorem{defin}[thm]{Definition}
\newtheorem{rem}[thm]{Remark}
\newtheorem{exa}[thm]{Example}

\numberwithin{equation}{section}

\begin{document}

%%%%% To ease editing, for IMPAN journals add:

\baselineskip=17pt

%%%%%%%%%%%

%% In the running head, replace first names by initials
%% and give an abbreviation of the title.

\title[On Baire-one mappings]{On Baire-one mappings \\with zero-dimensional domains}

\author[O. Karlova]{Olena Karlova}
\address{Department of Mathematical Analysis\\ Faculty of Mathematics and Informatics\\ Chernivtsi National University\\ Chernivtsi, Ukraine}
\email{maslenizza.ua@gmail.com}

\date{}

\begin{abstract}
We generalize the Lebesgue-Hausdorff Theorem on Baire classification of mappings defined on strongly zero-dimensional spaces.
\end{abstract}



\maketitle

\section{Introduction}

A subset $A$ of a topological space $X$ is  {\it functionally $F_\sigma$- ($G_\delta$)-set} if  $A$ is a union (an intersection) of a sequence of zero (cozero) subsets of $X$. If a set is functionally  $F_\sigma$ and functionally $G_\delta$ simultaneously, then it is called {\it functionally ambiguous}.

Let $X$ and $Y$ be topological spaces and $f:X\to Y$ be a mapping. We say that $f$ belongs to
\begin{itemize}
\item  {\it the first Baire class}, $f\in {\rm B}_1(X,Y)$, if $f$ is a pointwise limit of a sequence of continuous mappings between  $X$ and $Y$;

   \item {\it the first (functional) Lebesgue class}, $f\in {\rm H}_1(X,Y)$ ($f\in {\rm K}_1(X,Y)$), if $f^{-1}(V)$ is (functionally) $F_\sigma$-set in $X$ for any open subset $V$ of  $Y$.
\end{itemize}

Obviously, ${\rm H}_1(X,Y)={\rm K}_1(X,Y)$ for a perfectly normal space $X$ and a topological space $Y$. It is not hard to verify that the inclusion ${\rm B}_1(X,Y)\subseteq {\rm H}_1(X,Y)$ holds for any topological space  $X$ and a perfectly normal space $Y$ (see~\cite[p.~386]{Kuratowski:Top:1}).
But the proof of the inverse inclusion is much more difficult problem that begins in the PhD thesis of Ren\'{e} Baire~\cite{Baire1}.

The classical Lebesgue-Hausdorff theorem~\cite{Lebesgue:1905, Hausdorff:1957} tells that
\begin{equation}\label{eq:B1=H1}
{\rm B}_1(X,Y)={\rm H}_1(X,Y)
\end{equation}
if $X$ is  a metric space and $Y=[0,1]^{\omega}$, or if $X$ is a zero-dimensional metrizable separable space and $Y$ is a metrizable separable space (see~\cite[Theorem 24.10]{Kechris}). This result was generalized by many mathematicians in several ways. The first direction concerns the verification of~(\ref{eq:B1=H1}) for a connected-like space $Y$. So, the equality~(\ref{eq:B1=H1}) holds under the following assumptions:
\begin{enumerate}
  \item[(I)] $X$ is a metrizable space, $Y$ is a separable convex subset of a Banach space (S.~Rolewicz \cite{Rolewicz});

  \item[(II)] $X$ is normal, $Y=\mathbb R$ (M.~Laczkovich~\cite{Laczk} without proof);

  \item[(III)] $X$ is a complete metric space, $Y$ is a Banach space (C.~Stegall~\cite{Stegall}).
\end{enumerate}
Moreover, it was proved that if
\begin{enumerate}
  \item[(IV)] $X$ is a topological space and $Y$ is a metrizable separable arcwise connected and locally arcwise connected space (O.~Karlova, V.~Mykhaylyuk~\cite{Karlova-Mykhaylyuk:2006}),
\end{enumerate}
then
\begin{equation}\label{eq:B1=K1}
{\rm B}_1(X,Y)={\rm K}_1(X,Y).
\end{equation}

R.~Hansell in~\cite{Hansell:1971} (see also~\cite{Hansell:1974}) introduced the notion of $\sigma$-discrete mapping as a convenient tool for the investigation of Borel measurable mappings with valued in non-separable metric spaces. A mapping $f:X\to Y$ is called $\sigma$-discrete if there exists a family $\mathcal B=\bigcup\limits_{n=1}^\infty \mathcal B_n$ of subsets of a space $X$ such that every family $\mathcal B_n$ is discrete in $X$ and the preimage $f^{-1}(V)$ of any open set $V$ in $Y$ is a union of sets from $\mathcal B$. The class of all $\sigma$-discrete mappings between $X$ and $Y$ is denoted by $\Sigma(X,Y)$. It is easy to see that if $Y$ is metrizable separable then every mapping $f:X\to Y$ is $\sigma$-discrete. The equality
\begin{equation}\label{eq:B1=H1andSigma}
  {\rm B}_1(X,Y)={\rm H}_1(X,Y)\cap\Sigma(X,Y)
\end{equation}
holds in the following situations:
\begin{enumerate}
  \item[(V)] $X$ is metrizable, $Y$ is a convex subset of a normed space (R.~Hansell \cite{Hansell:1974});

  \item[(VI)] $X$ is collectionwise normal, $Y$ is a closed convex subset of a Banach space (R.~Hansell~\cite{Hansell:1992});

  \item[(VII)] $X$ is a metrizable space, $Y$ is a metrizable space, every continuous function from a closed subset of $X$ to $Y$ can be extended continuously on  $X$, and for each  $y\in Y$ and each neighborhood $V$ of  $y$ in $Y$ there exists a neighborhood  $W$ of $y$ such that each continuous function from a closed subset $F\subseteq X$ to $V$ admits an extension $f:X\to V$ (C.A.~Rogers~\cite{Rogers});

\item[(VIII)] $X$ is a perfectly normal paracompact space, $Y$ is a Banach space (J.E.~Jayne, J.~Orihuela, A.J.~Pallar\'{e}s, G.~Vera~\cite{JOPV});

\item[(IX)] $X$ is metrizable, $Y$ is metrizable arcwise connected and locally arcwise connected  (M.~Fosgerau~\cite{Fos}).
\end{enumerate}

L.~Vesel\'{y} in~\cite{Vesely} noticed that every Baire-one mapping $f$ between a topological space $X$ and a metrizable space $Y$ is  ''strongly $\sigma$-discrete'', i.e., there exists a family $\mathcal B=\bigcup\limits_{n=1}^\infty\mathcal B_n$ of subsets of $X$ such that for every family  $\mathcal B_n=(B_i:i\in I_n)$ there exists a discrete family $(U_i:i\in I_n)$ of open sets in $X$ with $\overline{B_i}\subseteq U_i$ for all  $i\in I_n$ and, moreover, the preimage $f^{-1}(V)$ of any open set $V$ in $Y$ is a union of sets from $\mathcal B$. The collection of all such mappings Vesel\'{y} denoted by $\Sigma^*(X,Y)$ and proved the equality
\begin{equation}\label{eq:B1=H1andSigmaStar}
{\rm B}_1(X,Y)={\rm H}_1(X,Y)\cap\Sigma^*(X,Y),
\end{equation}
in particular, in the case when
\begin{enumerate}
\item[(X)] $X$ is a normal space, $Y$ is a metrizable arcwise connected and locally arcwise connected space (L. Vesel\'{y}~\cite{Vesely}).
\end{enumerate}

The second way of the development of the Lebesgue-Hausdorff theorem deals with the case when $Y$ does not satisfy any properties like connectedness, but $X$ is zero-dimensional, strongly zero-dimensional, etc. In this direction the following results were obtained: the equality~(\ref{eq:B1=H1}) holds if
\begin{enumerate}
  \item[(XI)] $X$ is a normal strongly zero-dimensional space, $Y$ is a zero-dimensional metrizable separable space (H.~Shatery, J.~Zafarani~\cite{ShZ});
\end{enumerate}
the equality~(\ref{eq:B1=K1}) takes place when
\begin{enumerate}
  \item[(XII)] $X$ is a strongly zero-dimensional space, $Y$ is a metrizable separable space (O.~Karlova~\cite{Karlova:UMV:2007});
\end{enumerate}
and the equality~(\ref{eq:B1=H1andSigma}) is valid if
\begin{enumerate}
  \item[(XIII)] $X$ is a strongly zero-dimensional metrizable space and $Y$ is a metrizable space (O.~Karlova~\cite{Karlova:Visnyk:2008}).
\end{enumerate}

Finally, the third direction is connected with the case when $Y$ is non-metrizable. Here we are able to prove the equality~(\ref{eq:B1=H1}) if
\begin{enumerate}
  \item[(XIV)] $X$ is a hereditarily Baire separable metrizable space, $Y$ is a strict inductive limit of a sequence of metrizable
locally convex spaces (O.~Karlova, V.~Mykhaylyuk~\cite{Karlova-Mykhaylyuk:2006:MS}).
\end{enumerate}
However, it remains here many unsolved problems, in particular, the following.
\begin{question}\label{q:OpenProblems:1}\cite[Question 3.3, p.~659]{OpenProblems2}
  Does every ${\rm H}_1$-mapping $f:[0,1]\to C_p[0,1]$ belong to the first Baire class?
\end{question}
This question is equivalent to the following one.
\begin{question}\label{q:OpenProblems:2}\cite[Question 3.4, p.~659]{OpenProblems2}
 Let $f:[0,1]\times [0,1]\to \mathbb R$  be a function which is continuous with respect to the first variable and belongs to the first Baire class with respect to the second one. Is $f$ a pointwise limit of a sequence of separately continuous functions?
\end{question}

In this paper we develop technics from \cite{Fos} and generalize the Lebesgue-Hausdorff Theorem for $\sigma$-discrete mappings defined on strongly zero-dimen\-sional spaces with valued in metrizable spaces. In order to do this we consider the class of $\sigma$-strongly functionally discrete mappings introduced in \cite{Karlova:TA:2015}. We denote this class by $\Sigma^f(X,Y)$ and notice that $\Sigma^f(X,Y)=\Sigma^*(X,Y)$ if $X$ is a normal space. We prove that 
${\rm K}_1(X,Y)\cap\Sigma^f(X,Y)={\rm B}_1(X,Y)$, if $X$ is a strongly zero-dimensional space and $Y$ is a metrizable space.  We also introduce almost strongly zero-dimensional spaces and prove that
if $X$ is a topological space and $Y$ is a disconnected metrizable separable space, then the following conditions are equivalent: (i) $X$ is almost strongly zero-dimensional; (ii) ${\rm K}_1(X,Y)={\rm B}_1(X,Y)$.

\section{Relations between functionally $\sigma$-discrete and ${\rm B}_1$-mappings}

\begin{defin}
 {\rm A family $\mathcal A=(A_i:i\in I)$ of subsets of a topological space $X$ is said to be
  \begin{enumerate}
    \item {\it discrete} if every point of $X$ has a neighborhood which intersects at most one set from the family $\mathcal A$;

    \item {\it strongly discrete} if there exists a discrete family $(U_i:i\in I)$ of open subsets of $X$ such that $\overline{A_i}\subseteq U_i$ for every $i\in I$;

    \item {\it strongly functionally discrete} or, briefly, {\it sfd-family} if there exists a discrete family $(U_i:i\in I)$ of cozero subsets of $X$ such that $\overline{A_i}\subseteq U_i$ for every $i\in I$.   
  \end{enumerate}}
\end{defin}

\begin{rem}\label{remark:discr_families}
$\phantom{a}$
{\rm  \begin{enumerate}
    \item For an arbitrary space  $X$ we have  (3) $\Rightarrow$ (2) $\Rightarrow$ (1);

    \item $X$ is collectonwise normal if and only if  (1)=(2);

    \item if $X$ is normal, then (2)=(3).\label{rem:normal}
    \end{enumerate}}
\end{rem}

\begin{defin}
{\rm Let $\mathcal P$ be a property of a family of sets. A family $\mathcal A$ is called {\it a $\sigma$-$\mathcal P$ family} if  $\mathcal A=\bigcup\limits_{n=1}^\infty {\mathcal A}_n$, where every family $\mathcal A_n$ has the property $\mathcal P$.}
\end{defin}

\begin{defin}\label{def:base-for-a-function}
{\rm A family $\mathcal B$ of sets of a topological space $X$ is called {\it a base}  for a mapping $f:X\to Y$ if the preimage $f^{-1}(V)$ of an arbitrary open set  $V$ in $Y$ is a union of sets from $\mathcal B$.}
\end{defin}

Clearly, we may assume that $V$ is an element of an open base of $Y$ in Definition~\ref{def:base-for-a-function}.

\begin{defin}
{\rm  If a mapping $f:X\to Y$ has a base which is a $\sigma$-$\mathcal P$ family, then we say that $f$ is {\it a $\sigma$-$\mathcal P$ mapping}.}
\end{defin}

The collection of all $\sigma$-$\mathcal P$ mappings between $X$ and $Y$ we will denote by
\begin{itemize}
  \item $\Sigma(X,Y)$ if $\mathcal P$ is a property of discreteness;

  \item $\Sigma^*(X,Y)$ if $\mathcal P$ is a property of a strong discreteness;

  \item $\Sigma^f(X,Y)$ if $\mathcal P$ is a property of a strong functional discreteness;
  
   \item  $\Sigma^{f}_0(X,Y)$ if $f$ has a $\sigma$-sfd base of zero sets.
\end{itemize}

Let us observe that a continuous mapping $f:X\to Y$ is $\sigma$-discrete if either $X$ or $Y$ is a metrizable space, since every metrizable space has a $\sigma$-discrete base of open sets. Moreover, it is evident that every mapping with values in a second countable space is $\sigma$-discrete. In~\cite{Hansell:1971} Hansell proved that every Borel measurable mapping  $f:X\to Y$ between  a complete metric space $X$ and  a metric space $Y$ is $\sigma$-discrete. For any metric spaces $X$ and $Y$ the family $\Sigma(X,Y)$ is closed under pointwise limits~\cite{Hansell:1987}, which implies that every Baire measurable mapping between metric spaces is $\sigma$-discrete.

The following fact implies from \cite[Theorem 6]{Karlova:TA:2015}. 
\begin{thm}\label{cor:sigma_F}
Let $X$ be a topological space and $Y$ be a metrizable space. Then
${\rm K}_1(X,Y)\cap\Sigma^f(X,Y)=\Sigma^f_0(X,Y)$. 
\end{thm}

The next simple lemma will be useful.

\begin{lem}\label{union_SFD}
Let $X$ be a topological space, $(U_i:i\in I)$ be a locally finite family of cozero subsets of $X$, $(F_i:i\in I)$ be a family of zero subsets of $X$ such that $F_i\subseteq U_i$ for every $i\in I$. Then $F=\bigcup\limits_{i\in I} F_i$ is a zero set in $X$.
\end{lem}

\begin{proof}
For every $i\in I$ we choose a continuous function  $f_i:X\to [0,1]$ such that
$F_i=f_i^{-1}(0)$ and $X\setminus U_i=f_i^{-1}(1)$. For every $x\in X$ let $f(x)=\min\limits_{i\in I} f_i(x)$. Then $f:X\to [0,1]$ is continuous and $F=f^{-1}(0)$.
\end{proof}

\begin{cor}\label{cor:unionsfd}
 A union of an sfd-family of zero sets in a topological space is a zero set.
\end{cor}

We say that a topological space $X$ {\it is strongly zero-dimensional} if for any completely separated subsets $A$ and $B$ of $X$ there exists a clopen  set  $U$ such that $A\subseteq U\subseteq X\setminus B$.

For families $\mathcal A$ and $\mathcal B$ we write $\mathcal A \prec \mathcal B$ if for every $A\in\mathcal A$ there exists $B\in\mathcal B$ such that $A\subseteq B$.

\begin{prop}\label{prop:main_technical}
Let $X$ be a  strongly zero-dimensional space, $(Y,d)$ be a metric space, \mbox{$f:X\to Y$} be a mapping,
 ${\mathcal F}_1,\dots,{\mathcal F}_n$ be families of zero subsets of $X$ such that
\begin{enumerate}
\item $\mathcal F_k$ is an sfd-family for every $k=1,\dots,n$;

\item  ${\mathcal F}_{k+1}\prec {\mathcal F}_k$ for every $k=1,\dots,n-1$;

\item  for every  $k=1,\dots,n$ the inequality ${\rm diam}(f(F))<\frac{1}{2^{k+2}}$ holds for all $F\in {\mathcal F}_k$.
\end{enumerate}
Then there exists a continuous mapping $g:X\to Y$ such that the inclusion $x\in \cup {\mathcal F}_k$ for some   $k=1,\dots,n$ implies the inequality
\begin{gather}\label{prop:main_technical_ineq}
d(f(x),g(x))<\frac{1}{2^k}.
\end{gather}
\end{prop}

\begin{proof}  Let $\mathcal F_k=(F_{i,k}:i\in I_k)$, $k=1,\dots,n$. We choose a discrete family $(U_{i,1}:i\in I_1)$ of cozero sets in  $X$ such that  $F_{i,1}\subseteq U_{i,1}$ for every $i\in I_1$. Let $(V_{i,1}:i\in I_1)$ be a family of clopen sets such that  $F_{i,1}\subseteq V_{i,1}\subseteq U_{i,1}$. Now we take a discrete family $(G_{i,2}:i\in I_2)$ of cozero sets such that $F_{i,2}\subseteq G_{i,2}$  for every $i\in I_2$. Since $\mathcal F_2\prec \mathcal F_1$, for every $i\in I_2$ there exists unique $j\in I_1$ such that $F_{i,2}\subseteq F_{j,1}$. We denote $U_{i,2}=G_{i,2}\cap V_{j,1}$ and choose a clopen set $V_{i,2}$ with $F_{i,2}\subseteq V_{i,2}\subseteq U_{i,2}\subseteq V_{j,1}$. Proceeding in this way we obtain discrete families $(U_{i,k}:i\in I_k)$ and $(V_{i,k}:i\in I_k)$ of subsets of $X$ for $k=1,\dots,n$ such that
$U_{i,k}$ is a cozero set, $V_{i,k}$ is a clopen set,  for every $k=1,\dots,n-1$  and $i\in I_{k+1}$ there exists unique $j\in I_k$ such that
  \begin{equation}\label{prop:main_technical_cond1}
  F_{i,k+1}\subseteq F_{j,k},
  \end{equation}
  \begin{equation}\label{prop:main_technical_cond2}
   F_{i,k+1}\subseteq V_{i,k+1}\subseteq U_{i,k+1}\subseteq V_{j,k}.
  \end{equation}
Observe that for every $k$ the set $V_k=\bigcup\limits_{i\in I_k}V_{i,k}$ is clopen by Corollary~\ref{cor:unionsfd}.

Let $y_0\in f(X)$ and $y_{i,k}\in f(F_{i,k})$ be arbitrary points for every $k$ and $i\in I_k$. For all $x\in X$ let
$$
g_0(x)=y_0.
$$
Suppose that for some  $k$, $1\le k<n$, we have defined continuous mappings $g_1,\dots,g_k$ such that
\begin{gather}\label{c4}
g_k(x)=\left\{\begin{array}{ll}
                g_{k-1}(x), & \mbox{if\,\,\,}x\in X\setminus V_k, \\
                y_{i,k}, & \mbox{if\,\,\,}x\in V_{i,k}\mbox{\,\,for some\,\,} i\in I_k.
              \end{array}
\right.
\end{gather}
Let
$$
g_{k+1}(x)=\left\{\begin{array}{ll}
                g_{k}(x), & \mbox{if\,\,\,}x\in X\setminus V_{k+1}, \\
                y_{i,k+1}, & \mbox{if\,\,\,}x\in V_{i,k+1}\mbox{\,\,for some\,\,} i\in I_{k+1}.
              \end{array}
\right.
$$
Then the mapping $g_{k+1}:X\to Y$ is continuous, since every restriction $g_{k+1}|_{V_{k+1}}$ and $g_{k+1}|_{X\setminus V_{k+1}}$  is continuous and the set $V_{k+1}$ is clopen. Proceeding inductively we define continuous mappings $g_1,\dots,g_n$ satisfying~(\ref{c4}).

We put $g=g_n$ and prove that $g$ satisfies~(\ref{prop:main_technical_ineq}). We first show that
\begin{gather}\label{c3}
  d(g_{k+1}(x),g_k(x))<\frac{1}{2^k}
\end{gather}
for all $0\le k<n$ and $x\in X$. Indeed, if $x\in X\setminus V_{k+1}$, then $g_{k+1}(x)=g_k(x)$ and $ d(g_{k+1}(x),g_k(x))=0$. Assume $x\in V_{i,k+1}$ for some $i\in I_{k+1}$. Take $j\in I_k$ such that conditions~(\ref{prop:main_technical_cond1}) and~(\ref{prop:main_technical_cond2}) hold. Then $g_{k+1}(x)=y_{i,k+1}$ and $g_k(x)=y_{j,k}$. Since $f(F_{i,k+1})\subseteq f(F_{j,k})$,  $y_{i,k+1}\in f(F_{j,k})$. Hence, $d(g_{k+1}(x), g_k(x))\le {\rm diam}(f(F_{j,k}))<\frac{1}{2^{k+2}}$.

Let $1\le k\le n$ and $x\in \cup\mathcal F_k$. Then $x\in F_{i,k}\subseteq V_{i,k}$ for some $i\in I_k$. It follows that $g_k(x)=y_{i,k}$ and $d(f(x),g_k(x))\le {\rm diam}(f(F_{i,k}))<\frac{1}{2^{k+2}}$. Taking into account~(\ref{c3}), we obtain that
\begin{gather*}
  d(f(x),g(x))\le d(f(x),g_k(x))+\sum\limits_{i=k}^{n-1} d(g_i(x),g_{i+1}(x))<\frac{1}{2^{k+2}}+\frac{1}{2^{k+1}}<\frac{1}{2^k}.
\end{gather*}
\end{proof}

\begin{thm}\label{Sigma_is_BaireOne}
   Let $X$  be a strongly zero-dimensional space and $Y$ be a metrizable space. Then
   $\Sigma_0^f(X,Y)\subseteq {\rm B}_1(X,Y)$.   
\end{thm}

\begin{proof} Fix a metric $d$ on $Y$ which generates its topological structure. For every $k\in\mathbb N$ we consider  a covering  $\mathcal U_k$  of $Y$ by open sets  with diameters at most $\frac{1}{2^k}$. Let $f\in \Sigma_0^f(X,Y)$ and $\mathcal B$ be a $\sigma$-sfd base for $f$, which consists of zero subsets of $X$. For every $k\in\mathbb N$ we put
$$
\mathcal B_k=(B\in\mathcal B: \exists U\in \mathcal U_k \,\,\,|\,\,\, B\subseteq f^{-1}(U)).
$$
Then $\mathcal B_k$ is a $\sigma$-sfd family and $X=\cup \mathcal B_k$ for every $k$. According to \cite[Lemma 13]{Karlova:TA:2015} for every  $k\in\mathbb N$ there exists a sequence $(\mathcal B_{k,n})_{n=1}^\infty$ of sfd families of zero subsets of $X$ such that $\mathcal B_{k,n}\prec\mathcal \mathcal B_k$, $\mathcal B_{k,n}\prec\mathcal B_{k,n+1}$  for every $n\in\mathbb N$ and $\bigcup\bigcup\limits_{n=1}^\infty \mathcal B_{k,n}=X$.
For all $k,n\in\mathbb N$ we set
$$
\mathcal F_{k,n}=(B_1\cap\dots\cap B_k: B_m\in \mathcal B_{m,n}, 1\le m\le k).
$$
Notice that each of the families $\mathcal F_{k,n}$ is strongly functionally discrete, consists of zero sets and satisfy the following conditions:
\begin{enumerate}
  \item[(a)] $\mathcal F_{k+1,n}\prec \mathcal F_{k,n}$,

  \item[(b)] $\mathcal F_{k,n}\prec \mathcal F_{k,n+1}$,

  \item[(c)]  $\bigcup\limits_{n=1}^\infty \mathcal F_{k,n}=X$.
\end{enumerate}
For every $n\in\mathbb N$ we apply Proposition~\ref{prop:main_technical} to $f$ and to the families $\mathcal F_{1,n}$, $\mathcal F_{2,n}$,\dots, $\mathcal F_{n,n}$. We get a sequence of continuous mappings  $g_n:X\to Y$ such that the inclusion $x\in \cup\mathcal F_{k,n}$ for some  $k\le n$ implies
$d(f(x),g_n(x))<\frac{1}{2^k}$. It is easy to see that properties~(b) and~(c) imply that  $g_n\to f$ pointwise on $X$. Hence, $f\in {\rm B}_1(X,Y)$.
\end{proof}

\begin{prop}\label{BaireOne_is_SFD}
  Let $X$ be a topological space and $Y$ be a metrizable space. Then
  ${\rm B}_1(X,Y)\subseteq \Sigma^f_0(X,Y)$.
\end{prop}

\begin{proof}
  Let $f\in {\rm B}_1(X,Y)$ and $(f_n)_{n=1}^\infty$ be a sequence of continuous mappings $f_n:X\to Y$ such that  $f(x)=\lim\limits_{n\to\infty} f_n(x)$ for all $x\in X$. Let $\mathcal V=\bigcup\limits_{m=1}^\infty \mathcal V_m$ be a $\sigma$-discrete open base of the space $Y$. For every $V\in \mathcal V$  we choose a sequence $(G_{k,V})_{k=1}^\infty$ of open sets such that $\overline{G_{k,V}}\subseteq G_{k+1,V}$ for every $k\in\mathbb N$ and $V=\bigcup\limits_{k=1}^\infty \overline{G_{k,V}}$. It is not hard to verify that
  \begin{gather}\label{gath:B1isK1}
  f^{-1}(V)=\bigcup\limits_{k=1}^\infty\bigcap\limits_{n=k}^{\infty} f_n^{-1}(\overline{G_{k,V}}).
  \end{gather}
Denote $F_{k,V}=\bigcap\limits_{n=k}^{\infty} f_n^{-1}(\overline{G_{k,V}})$ and notice that every  $F_{k,V}$ is a zero set in $X$. For all $k,m\in\mathbb N$ we put $\mathcal B_{k,m}=(F_{k,V}:V\in\mathcal V_m)$ and  $\mathcal B=\bigcup\limits_{k,m=1}^\infty\mathcal B_{k,m}$. Then $\mathcal B$ is a base for $f$. Moreover, every family $\mathcal B_{k,m}$ is strongly functionally discrete, since $F_{k,V}\subseteq f_k^{-1}(V)$ and the family $(f_k^{-1}(V):V\in\mathcal V_m)$ is discrete and consists of cozero sets. 
\end{proof}

Combining Theorems~\ref{cor:sigma_F},~\ref{Sigma_is_BaireOne} and Proposition~\ref{BaireOne_is_SFD}, we get

\begin{thm}\label{thm:main1}
Let $X$ be a strongly zero-dimensional space and $Y$ be a metrizable space. Then
 ${\rm K}_1(X,Y)\cap\Sigma^f(X,Y)={\rm B}_1(X,Y)$.
\end{thm}

According to Theorem 3 from~\cite{Hansell:1971} we have ${\rm K}_1(X,Y)\subseteq \Sigma^f(X,Y)$ for any completely metrizable $X$ and metrizable~$Y$. This fact and Theorem~\ref{thm:main1} immediately imply the following result.

\begin{thm}\label{thm:complete_case}
  Let $X$ be a completely metrizable strongly zero-dimensional space and $Y$ be a metrizable space. Then
 ${\rm K}_1(X,Y)={\rm B}_1(X,Y)$.
\end{thm}

We show that the metrizability of $Y$ in Theorem~\ref{thm:complete_case} is essential.

\begin{exa}
  There exists a  completely metrizable strongly zero-dimensional space  $X$ and a Lindel\"{o}f strongly zero-dimensional space $Y$ such that ${\rm K}_1(X,Y)\setminus {\rm B}_1(X,Y)\ne\emptyset$.
\end{exa}

\begin{proof}
   Let  $X$ be the set of all irrational numbers  with the euclidian topology and $Y$ be the same set with the topology induced from the Sorgenfrey line (recall that the Sorgenfrey line is the real line $\mathbb R$ endowed with the topology generated by the base consisting of all semi-intervals $[a,b)$, where $a<b$).

Take a countable dense in $X$ set $Q=\{q_n:n\in {\mathbb N}\}$ and for all $x\in X$ we put
$$
f(x)=\left\{\begin{array}{ll}
  x-\frac{1}{n}, & \mbox{if\,\,\,} x=q_n, \\
  x, & \mbox{if\,\,\,}x\in X\setminus Q.
\end{array}
\right.
$$

To show that $f\in {\rm K}_1(X,Y)$ it is enough to verify that $f^{-1}([a,b)\cap Y)$ is an $F_\sigma$-set in $X$ for every $a<b$, since $Y$ is Lindel\"{o}f. Denote $E=f^{-1}([a,b)\cap Y)$. Notice that the sets $A=(X\cap [a,b))\setminus E$ and $B=E\setminus (a,b)$ are subsets of $Q$.
Let $(n_k)_{k=1}^\infty$ be an increasing sequence of numbers such that $A=\{q_{n_k}:k\in\mathbb N\}$. Then $a\le r_{n_k}\le a+\frac{1}{n_k}$ for every $k$.
Therefore, $\lim\limits_{k\to\infty}r_{n_k}=a$ in $X$. Hence, the set $A$ is $G_\delta$ in $X$. Then the equality  $E=(X\cap [a,b)\setminus A)\cup B$ implies that $E$ is an $F_\sigma$-set in $X$.

Now we prove that $f\not\in {\rm B}_1(X,Y)$. Assume that there exists a sequence of continuous mappings $f_n:X\to Y$ such that $f_n(x)\to f(x)$ for every $x\in X$. Let
$$
A_n=\{x\in X: \forall k\ge n\,\,\, f_k(x)\ge f(x)\}
$$
for $n\in\mathbb N$. It is easy to see that $\bigcup\limits_{n=1}^\infty A_n=X$. Since the set $X\setminus Q$  is of the second category in $X$, there exist a number $n$ and a set $[a,b]$ such that $[a,b]\cap X\subseteq \overline{A_n\setminus Q}$.
Notice that $A_n\setminus Q\subseteq F$, where $F=\bigcap\limits_{k=n}^\infty\{x\in X: f_k(x)\ge x\}$. Since $F$ is closed in $X$, we have
$[a,b]\cap X\subseteq F$. Then  $f(x)=\lim\limits_{k\to\infty} f_k(x)\ge x$ for all $x\in [a,b]\cap X$, a contradiction.
\end{proof}

\section{Almost strongly zero-dimensional spaces and characterization theorems}

In this section we find necessary conditions on a space $X$  under which the equality ${\rm K}_1(X,Y)\cap\Sigma^f(X,Y)={\rm B}_1(X,Y)$ holds for any disconnected metrizable space $Y$.

\begin{defin} {\rm A subset  $F$ of a topological space $X$ is called {\it a $C$-set} if there exists a sequence $(U_n)_{n=1}^\infty$ of clopen sets in $X$ such that $F=\bigcap\limits_{n=1}^\infty U_n$. A set is called {\it a $C_\sigma$-set} if it is a union of a sequence of $C$-sets.}
\end{defin}

\begin{defin}
{\rm We say that a topological space $X$ is {\it almost strongly zero-dimensional} if every zero subset of $X$ is a $C_\sigma$-set.}
\end{defin}

Notice that  every strongly zero-dimensional space is almost strongly zero-dimensional.

\begin{lem}\label{lemma:separated_C_sets}
  Let $X$ be a topological space, $C_1$ and $C_2$ be disjoint $C$-subsets of $X$. Then there exists a clopen set  $G$ in $X$ such that $C_1\subseteq G\subseteq X\setminus C_2$.
\end{lem}

\begin{proof}
  Let $(U_n)_{n=1}^\infty$ and $(V_n)_{n=1}^\infty$ be sequences of clopen subsets of  $X$ such that $X\setminus C_1=\bigcup\limits_{n=1}^\infty U_n$ and $X\setminus C_2=\bigcup\limits_{n=1}^\infty V_n$. For every $n\in\mathbb N$ put $G_n=V_n\setminus \bigcup\limits_{k=1}^n U_k$ and let $G=\bigcup\limits_{n=1}^\infty G_n$. Clearly, $C_1\subseteq G\subseteq X\setminus C_2$ and $G$ is open in $X$. It remains to show that $G$ is closed. Let  $x\in\overline G$. If $x\in C_1$ then $x\in G$. If $x\not\in C_1$, then there is $N\in\mathbb N$ such that $x\in U_N$. Notice that $U_N\cap G_n=\emptyset$ for all $n\ge N$. Then  $x\in \overline{\bigcup\limits_{n=1}^{N-1}G_n}=\bigcup\limits_{n=1}^{N-1}\overline{G_n}=\bigcup\limits_{n=1}^{N-1} G_n\subseteq G$.
\end{proof}

\begin{cor}\label{cor:separated_n-C-sets}
  Let $X$ be a topological space and $C_1,\dots,C_n$ be disjoint  $C$-subsets of $X$, $n\in\mathbb N$. Then there exist disjoint clopen sets  $G_1,\dots,G_n$ in $X$ such that $X=G_1\cup\dots\cup G_n$ and $C_i\subseteq G_i$ for every  $i=1,\dots,n$.
\end{cor}

\begin{prop}\label{prop:aszd_is_totally_sep}
 Every almost strongly zero-dimensional completely regular space $X$ is totally separated.
\end{prop}

\begin{proof}
  Let $x,y\in X$ be distinct points and $U,V$ be disjoint zero neighborhoods of $x$ and $y$, respectively. Since $X$ is almost strongly zero-dimensional, there exist $C$-sets $C_x$ and $C_y$ such that $x\in C_x\subseteq U$ and $y\in C_y\subseteq V$. By Lemma~\ref{lemma:separated_C_sets} there exists a clopen set $G$ such that $C_x\subseteq G$ and $G\cap C_y=\emptyset$. Hence, $x$ and $y$ can be separated by a clopen set which implies that $X$ is totally separated.
\end{proof}

\begin{lem}\label{lemma:separated_count_comp_and_C}
Let $X$ be a topological space, $F\subseteq X$ be a countably compact $C_\sigma$-set,  $C\subseteq X$ be a $C$-set and $F\cap C=\emptyset$. Then there exists a clopen set  $G$ in $X$ such that $F\subseteq G\subseteq X\setminus C$.
\end{lem}

\begin{proof}
Let $(C_n)_{n=1}^\infty$ be an increasing sequence of $C$-sets such that $F=\bigcup\limits_{n=1}^\infty C_n$. Lemma~\ref{lemma:separated_C_sets} implies that for every $n$ there exists a clopen set $G_n$ in $X$ such that  $C_n\subseteq G_n\subseteq X\setminus C$. Since $F$ is countably compact, we choose a finite subcovering $\mathcal G$ of the covering $(G_n:n\in\mathbb N)$ of $F$. It remains to put $G=\cup\mathcal G$.
\end{proof}

In the same manner we can prove the following result.

\begin{lem}\label{lemma:separated_count_comp_and_count_comp}
Let $X$ be a topological space and $F,E\subseteq X$ be disjoint countably compact $C_\sigma$-sets. Then there exists a clopen set $G$ in $X$ such that  $F\subseteq G\subseteq X\setminus E$.
\end{lem}

Taking into account that every closed subset of a countably compact space is countably compact, we obtain the following corollary from Lemma~\ref{lemma:separated_count_comp_and_count_comp}.

\begin{prop}\label{thm:ASZD_equiv}
Let $X$ be a countably compact space. Then $X$ is almost strongly zero-dimensional if and only if it is strongly zero-dimensional.
\end{prop}

The following question is open.
\begin{question}
  Do there exists a completely regular almost dimensional space which is not strongly zero-dimensional?
\end{question}

\begin{prop}\label{prop:K1_in_B1_implies_ASZD}
  Let $X$ be a topological space, $Y$ be a disconnected space such that ${\rm K}_1(X,Y)\cap\Sigma^f(X,Y)\subseteq {\rm B}_1(X,Y)$. Then $X$ is almost strongly zero-dimensional.
\end{prop}

\begin{proof}
Let $U$ and $V$ be clopen disjoint nonempty subsets of $Y$ such that $Y=U\cup V$, $F\subseteq X$ be a zero set, $y_1\in U$, $y_2\in V$ and let $f:X\to Y$ be a mapping such that $f(x)=y_1$ for all $x\in F$ and $f(x)=y_2$  for all $x\in X\setminus F$. It is easy to see that $f\in {\rm K}_1(X,Y)\cap \Sigma^f(X,Y)$. Then there exists  a sequence  $(f_n)_{n=1}^\infty$ of continuous mappings $f_n:X\to Y$ such that $\lim\limits_{n\to\infty}f_n(x)=f(x)$ for every $x\in X$. Then $F=f^{-1}(U)=\bigcup\limits_{n=1}^\infty\bigcap\limits_{m=n}^\infty f_m^{-1}(U)$. Hence, $F$ is a  $C_{\sigma}$-set.
\end{proof}

\begin{thm}
  Let $Y$ be a disconnected metrizable space. If
  \begin{enumerate}
    \item[(a)] $X$ is locally compact paracompact Hausdorff space, or

    \item[(b)] $X$ is a countably compact  space,
  \end{enumerate}
  then the following conditions are equivalent:
  \begin{enumerate}
    \item ${\rm K}_1(X,Y)\cap\Sigma^f(X,Y)={\rm B}_1(X,Y)$;

    \item $X$ is a strongly zero-dimensional space.
  \end{enumerate}
\end{thm}

\begin{proof} (1)$\Rightarrow$(2). According to Proposition~\ref{prop:K1_in_B1_implies_ASZD}, $X$ is almost strongly zero-dimensional. It follows that $X$ is strongly zero-dimensional in case (b) by Proposition~\ref{thm:ASZD_equiv}. In case (a) $X$ is completely regular and, consequently, totally separated by Proposition~\ref{prop:aszd_is_totally_sep}. It remains to apply Theorem 6.2.10 from~\cite{Eng}.

The implication (2)$\Rightarrow$(1) follows from Theorem~\ref{thm:main1}.
\end{proof}

A sequence $(f_n)_{n=1}^\infty$ of mappings $f_n:X\to Y$ is called {\it stably convergent} to a mapping $f:X\to
Y$ if for every $x\in X$ there exists a number $n_0$ such that
$f_n(x)=f(x)$ for all $n\ge n_0$. We denote this fact by $f_n{\stackrel{st}\longrightarrow}f$.

\begin{lem}\label{lemma:stable_finite_valued}
Let $X$ be an almost zero-dimensional space, $Y$ be a $T_1$-space and $f\in {\rm K}_1(X,Y)$ be a finite-valued mapping. Then there exists a sequence of continuous finite-valued mappings $f_n:X\to Y$ which is stably convergent to  $f$ on $X$.
\end{lem}

\begin{proof}
Denote $f(X)=\{y_1,\dots,y_m\}$. Since for every  $1\le i\le m$ the set $A_i=f^{-1}(y_i)$  is functionally $F_\sigma$ in  $X$, there exists an increasing sequence $(C_{i,n})_{n=1}^\infty$ of $C$-subsets of $X$ such that
$A_i=\bigcup\limits_{n=1}^\infty C_{i,n}$. According to Corollary~\ref{cor:separated_n-C-sets} for every $n\in\mathbb N$ there are disjoint clopen sets $G_{1,n},\dots,G_{m,n}$ such that $C_{i,n}\subseteq G_{i,n}$ for every $i=1,\dots,m$ and $X=G_{1,n}\cup\dots\cup G_{m,n}$. Now for every  $n\ge 1$ we put $f_n(x)=y_i$ if $x\in G_{i,n}$ for some $i=1,\dots,m$. It is easy to see that $f_n{\stackrel{st}\longrightarrow}f$ on $X$.
\end{proof}

\begin{lem}\label{lemma:uniform_limit}
Let $X$  be an almost zero-dimensional space, $(Y,d)$ be a metric space and $(f_n)_{n=1}^\infty$ be a sequence of finite-valued mappings  $f_n\in {\rm K}_1(X,Y)$  which is uniformly convergent to $f:X\to Y$. Then $f\in {\rm B}_1(X,Y)$.
\end{lem}

\begin{proof}
Without loss of generality we may assume that
\begin{gather}\label{gath:ineq1}
d(f_{n+1}(x),f_n(x))\le\frac{1}{2^{n+1}}
\end{gather}
for all $x\in X$ and $n\in\mathbb N$. By Lemma~\ref{lemma:stable_finite_valued} for every $n\in\mathbb N$ there exists a sequence $(f_{n,m})_{m=1}^\infty$ of continuous finite-valued mappings $f_{n,m}:X\to Y$ such that
\begin{gather}\label{gath:stable1}
  f_{n,m}{\stackrel{d}\longrightarrow} f_n.
\end{gather}

For all $x\in X$ and $m\in\mathbb N$ we put
\begin{gather*}
h_{0,m}(x)=h_{1,m}(x)=f_{1,m}(x).
\end{gather*}
Now assume that for some  $k\in\mathbb N$ we have already defined sequences $(h_{1,m})_{m=1}^\infty$, \dots, $(h_{k,m})_{m=1}^\infty$ of continuous finite-valued mappings such that
\begin{gather}\label{gath:stable2}
 h_{n,m}{\stackrel{d}\longrightarrow} f_n\quad \mbox{for all\,\,} n=1,\dots,k,
 \end{gather}
 \begin{gather}\label{gath:stable3}
  d(h_{n+1,m}(x),h_{n,m}(x))\le \frac{1}{2^{n+1}}\,\,\,\mbox{for all}\,\,\, x\in X, m\in\mathbb N\,\,\,\mbox{and}\,\,\,n=0,\dots,k-1.
\end{gather}
For every  $m\in\mathbb N$ let
\begin{gather*}
U_m=\bigl\{x\in X: d(f_{k+1,m}(x),h_{k,m}(x))\le\frac{1}{2^{k+1}}\bigr\}.
\end{gather*}
Then $U_m$ is clopen in $X$. Moreover, conditions~~(\ref{gath:ineq1}), (\ref{gath:stable1}) and (\ref{gath:stable2}) imply that
\begin{gather}\label{uni}
  X=\bigcup\limits_{n=1}^\infty\bigcup\limits_{m=n}^\infty U_m.
\end{gather}
Define a sequence of finite-valued continuous mappings $(h_{k+1,m})_{m=1}^\infty$ by the formula
$$
h_{k+1,m}(x)=\left\{\begin{array}{ll}
  f_{k+1,m}(x), & \,\,\mbox{if}\,\, x\in U_m, \\
  h_{k,m}(x), & \,\,\mbox{if}\,\,x\not\in U_m. \\
\end{array}
\right.
$$
Notice that~(\ref{gath:stable1}) and~(\ref{uni}) imply that $h_{k+1,m}{\stackrel{d}\longrightarrow} f_{k+1}$ on $X$.

We prove that the inequality~(\ref{gath:stable3}) holds. Fix $m\in\mathbb N$ and $x\in X$. If $x\in U_m$,  then
$h_{k+1,m}(x)=f_{k+1,m}(x)$ and $d(h_{k+1,m}(x),h_{k,m}(x))=d(f_{k+1,m}(x),h_{k,m}(x))\le\frac{1}{2^{k+1}}$. If $x\not\in U_m$, then $h_{k+1,m}(x)=h_{k,m}(x)$ i $d(h_{k+1,m}(x),h_{k,m}(x))=0$.

Finally, we show that $\lim\limits_{m\to\infty}h_{m,m}(x)=f(x)$. Let $x\in X$, $\varepsilon>0$ and $n_0$  be a number such that
$$
\frac{1}{2^{n_0}}<\frac{\varepsilon}{2}\quad\mbox {and}\quad
d(f_{n_0}(x),f(x))<\frac{\varepsilon}{2}.
$$
Take $m_0>n_0$ with $h_{n_0,m}(x)=f_{n_0}(x)$ for all $m\ge m_0$. Then
\begin{gather*}
d(h_{m,m}(x),f(x))\le \\\le \sum\limits_{i=n_0+1}^m d(h_{i-1,m}(x),h_{i,m}(x))+d(h_{n_0,m}(x),f_{n_0}(x))+d(f_{n_0}(x),f(x))<\\
<\sum\limits_{i=n_0+1}^m \frac{1}{2^i}+\frac{\varepsilon}{2}<\frac{1}{2^{n_0}}+\frac{\varepsilon}{2}<\varepsilon
\end{gather*}
for all  $m\ge m_0$. Hence, $f\in{\rm B}_1(X,Y)$.
\end{proof}

\begin{thm}\label{thm:ASZD_implies_K1_in_B1}
 Let $X$ be an almost zero-dimensional space and $Y$ be a metrizable separable space. Then ${\rm K}_1(X,Y)={\rm B}_1(X,Y)$.
\end{thm}

\begin{proof}
 The inclusion ${\rm B}_1(X,Y)\subseteq {\rm K}_1(X,Y)$ follows from the equality~(\ref{gath:B1isK1}).

Let $f\in {\rm K}_1(X,Y)$ and  $d$ be a metric on $Y$ such that $(Y,d)$ is completely bounded. For every $n\in{\mathbb N}$ we take a finite  $\frac 1n$-network $Y_n=\{y_{i,n}:i\in I_n\}$ in $Y$ and put $A_{i,n}=\{x\in X: d(f(x),y_{i,n})<\frac{1}{n}\}$ for  $n\in{\mathbb N}$ та $i\in I_n.$ Notice that for every $n$ the family  $(A_{i,n}:i\in I_n)$ is a covering of $X$ by functionally  $F_\sigma$-sets. Similarly as in the proof  of the Reduction Theorem~\cite[p.~350]{Kuratowski:Top:1} we take a sequence of disjoint functionally ambiguous sets $F_{i,n}$ in $X$ such that
$F_{i,n}\subseteq A_{i,n}$ and $\bigcup\limits_{i\in I_n}F_{i,n}=X$. For every $n\in\mathbb N$ we put $f_n(x)= y_{i,n}$ if $x\in F_{i,n}$ for some $i\in I_n$. Then $(f_n)_{n=1}^\infty$ is a sequence of finite-valued mappings $f_n\in {\rm K}_1(X,Y)$ which is uniformly convergent to $f$ on $X$. It remains to apply Lemma~\ref{lemma:stable_finite_valued}.
\end{proof}

Combining Proposition~\ref{prop:K1_in_B1_implies_ASZD} and Theorem~\ref{thm:ASZD_implies_K1_in_B1} we obtain the following result.
\begin{thm} If $X$ is a topological space and $Y$ is a disconnected metrizable separable space, then the following conditions are equivalent:
\begin{enumerate}
  \item $X$ is almost strongly zero-dimensional;

  \item ${\rm K}_1(X,Y)={\rm B}_1(X,Y)$.
\end{enumerate}
\end{thm}

%\subsection*{Acknowledgements}
%This research was partly supported by NSF (grant no. XXXX).

\normalsize
\baselineskip=17pt

%%%%%%%%%%%%%%%

\end{document}